\documentclass[reqno,final]{amsart}
\usepackage{natbib}  
\usepackage{fancyhdr} 
\usepackage{color} 
\usepackage{hyperref} 
\usepackage{graphicx} 

\usepackage{pstricks}
\usepackage{amsmath,amssymb,amsthm,amsfonts}
\usepackage{comment}

\definecolor{aleacolor}{rgb}{0.16,0.59,0.78}

\hypersetup{
breaklinks,
colorlinks=true,
linkcolor=aleacolor,
urlcolor=aleacolor,
citecolor=aleacolor}

\renewcommand{\headheight}{11pt}
\renewcommand{\cite}{\citet}

\theoremstyle{plain}
\newtheorem{thm}{Theorem}[section]                                          
\newtheorem{prop}[thm]{Proposition}                          
\newtheorem{lem}[thm]{Lemma}
\newtheorem{cor}[thm]{Corollary}

\theoremstyle{definition}

\theoremstyle{remark}

\newtheorem{xmpl}[thm]{Example}

\makeatletter \@addtoreset{equation}{section} \makeatother
\renewcommand\theequation{\thesection.\arabic{equation}}
\renewcommand\thefigure{\thesection.\arabic{figure}}
\renewcommand\thetable{\thesection.\arabic{table}}

\newcommand{\N}{\mathbb{Z}_{+}}
\newcommand{\Z}{\mathbb{Z}}

\begin{document}

\title[QSDs for the Voter Model on Complete Bipartite Graphs]{Quasi-Stationary Distributions for the Voter Model on Complete Bipartite Graphs}

\author{Iddo Ben-Ari}
\author{Hugo Panzo}
\author{Philip Speegle}
\author{R. Oliver VandenBerg}

\address{Department of Mathematics\newline
University of Connecticut\newline
Storrs, CT 06269, USA}
\email{iddo.ben-ari@uconn.edu}
\urladdr{\url{http://iddo.ben-ari.uconn.edu/wordpress/}}

\address{Faculties of Electrical Engineering and Mathematics\newline
Technion -- Israel Institute of Technology\newline
Haifa 32000, Israel}
\email{panzo@campus.technion.ac.il}
\urladdr{\url{https://sites.google.com/view/hugopanzo/}}

\address{University of Alabama\newline
Tuscaloosa, AL 35487, USA}
\email{pspeegle@crimson.ua.edu}

\address{Kenyon College\newline
Gambier, OH 43022, USA}
\email{vandenberg1@kenyon.edu}

\thanks{Research performed during \href{http://markov-chains-reu.math.uconn.edu}{Markov Chains REU}, partially supported by NSA grant H98230-19-1-0022 to Iddo Ben-Ari. Hugo Panzo was supported at the Technion by a Zuckerman Fellowship.}


\begin{abstract}
We consider the discrete-time voter model on complete bipartite graphs and study the quasi-stationary distribution (QSD) for the model as the size of one of the partitions tends to infinity while the other partition remains fixed. We show that the QSDs converge weakly to a nontrivial limit which features a consensus with the exception of a random number of dissenting vertices in the ``large'' partition. Moreover, we explicitly calculate the law of the number of dissenters and show that it follows the heavy-tailed Sibuya distribution with parameter depending on the size of the ``small'' partition. Our results rely on a discrete-time analogue of the well-known duality between the continuous-time voter model and coalescing random walks which we develop in the paper. 
\end{abstract}

\maketitle

\section{Introduction and Main Results}

The \emph{voter model} is an interacting particle system which models the evolution of opinions in a population of voters. In the classical version, the voters are vertices of $\Z^d$ which can hold the opinion ``0'' or ``1'' and the model evolves in continuous-time by having each vertex change its opinion at a rate proportional to the number of dissenting neighbors, see \cite{Liggett}. With these dynamics, it is clear that \emph{consensus} is an absorbing state and early investigations of the model were interested in the time required to reach consensus and characterization of nontrivial invariant measures. Besides regular lattices, the voter model has also been studied in discrete-time on heterogeneous graphs where it was shown in \cite{Redner_3} to display markedly different behavior, see also \cite{pull_model} where it is referred to as the \emph{asynchronous pull model}. Additionally, there are further extensions of the model which more closely approximate reality, see \cite{Redner_5} for a recent survey.

In this paper, we study the discrete-time voter model on the \emph{complete bipartite graphs} $K_{n,m}$. These are heterogeneous graphs whose vertex set can be partitioned into two groups, a ``large'' group $L$ of size $n$ and a ``small'' group $S$ of size $m$, where each vertex of $L$ is connected to all of the vertices of $S$ and vice versa, and there are no connections between vertices in the same group. While the time required to reach consensus in the voter model on $K_{n,m}$ and its behavior on the way to consensus have already been studied in \cite{Redner_3}, we investigate what happens when consensus is conditioned to never occur. More specifically, we study the \emph{quasi-stationary distribution} (QSD) for the voter model on $K_{n,m}$.

Loosely speaking, a QSD for a Markov chain with absorbing states is a stationary distribution for the chain conditioned on nonabsorption, see Section \ref{sec:QSD} for a precise definition. In the case of the voter model on $K_{n,m}$, one might ask what the distribution of opinions typically looks like if consensus hasn't been reached after a long time. Is the lack of consensus due to a small minority of dissenters or are the opinions relatively balanced? If $n\gg m$, does the distribution of opinions on $L$ differ qualitatively from that of $S$? In order to give a concise answer to these questions, we fix $m$ and find the limit in distribution of the sequence of QSDs as $n\to\infty$.  We point to the case of the voter model on the complete graphs $K_n$ in Example \ref{ex:complete_graph} as evidence that even the existence of a limiting distribution is not obvious.

Before stating our main results, we recall from \cite{Sibuya_1,compendium} that the \emph{Sibuya distribution} with parameter $\gamma\in (0,1]$ is a probability distribution on $\N=\{1,2,\dots\}$ with probability mass function $f_\gamma$ and probability generating function $\phi_\gamma$ given by

\begin{equation}\label{eq:fgamma}
\begin{split}
f_\gamma(k) &=\frac{\gamma}{k!}\prod_{j=1}^{k-1} (j-\gamma),~ k\in\N,\\
\phi_\gamma(z)& = 1-(1-z)^{\gamma},~|z|\leq 1.
\end{split}
\end{equation}
When $\gamma\in(0,1)$, the Sibuya distribution is heavy tailed and we note from \cite{Sibuya_5} that in this case $f_\gamma$ decays according to a power law with 
$$f_\gamma(k) \sim \frac{1}{\pi}\sin(\gamma \pi)\Gamma(1+\gamma)\frac{1}{k^{\gamma+1}}\text{ as }k\to\infty.$$
We denote this probability distribution by $\mathrm{Sib}(\gamma)$. See \cite{Sibuya_4,Sibuya_3,Sibuya_6} for some recent applications of the Sibuya distribution. We also use $\mathrm{Bern}(p)$ to denote a Bernoulli random variable which takes the values $1$ and $0$ with probability $p$ and $1-p$, respectively.
\begin{thm}
\label{th:main}
Let $C\sim \mathrm{Bern}(1/2)$ and $D\sim \mathrm{Sib}(\gamma_m)$ be independent, with $$\gamma_m =2 \left(1-\sqrt{1-\frac{1}{2m}}\right).$$
Then the distribution of opinions under the QSD for the voter model on $K_{n,m}$  as $n\to\infty$ converges weakly to the following:
\begin{enumerate}
    \item All vertices of $S$ have opinion $C$. 
    \item All but $D$ vertices in $L$ have opinion $C$. 
\end{enumerate}
\end{thm} 

\begin{cor}
The distribution of the number of disagreements along edges in $K_{n,m}$ under the QSD tends to $mD$, where $D\sim \mathrm{Sib}(\gamma_m)$. 
\end{cor}

The proof of Theorem \ref{th:main} will be given at the end of Section \ref{sec:main}. The rest of the paper is organized as follows. In Section \ref{sec:QSD} we recall some important facts from the theory of quasi-stationary distributions for finite state Markov chains. We describe the voter model on general finite graphs in Section \ref{sec:voter_model}. In Section \ref{sec:duality}, we develop for general finite graphs a discrete-time analogue of the well-known duality between the continuous-time voter model and coalescing random walks. We use this duality in Section \ref{sec:main} to compute explicitly the geometric tail of the time to consensus in the voter model on $K_{n,m}$ and subsequently prove our main results.

\section{Quasi-Stationary Distributions}
\label{sec:QSD}
Here we give a quick summary of the theory of quasi-stationary distributions for finite state Markov chains, see \cite{qsdbook} for more details. To this end, suppose that ${\bf Y}=(Y_t:t\in \Z_+)$ is a Markov chain on a finite state space ${\bar \Omega}$ with transition function ${\bar S}$. Recall that a state  $i$ is absorbing  if 
$$ P_{i}(Y_1 =i)=1.$$
We will assume 
\begin{enumerate} 
\item $\Delta$, the set of all absorbing states, is nontrivial:  it is not empty and its complement is not empty; and 
\item $\Delta$ is accessible from every state. That is, for every state $i$, there exists $t\in \Z_+$ such that $$ P_i (Y_t \in \Delta)>0.$$ 
\end{enumerate} 
Letting 
$$ \tau =\inf\{t\in \Z_+: Y_t \in \Delta\},$$ 
it follows from our assumptions that $\tau < \infty$ almost surely under any initial distribution. Let  $S$ denote the substochastic transition function obtained by restricting  ${\bar S}$ to the complement of $\Delta$ in ${\bar \Omega}$. We denote this complement by  $\Omega$. In other words, $S$ is the principal submatrix obtained from ${\bar S}$ by removing all rows and columns corresponding to states in $\Delta$. Let $\mu$ be an initial distribution for ${\bf Y}$ whose support is contained in $\Omega$. As there is no risk of ambiguity, we will abuse notation and also denote its restriction to $\Omega$ by $\mu$.  We have 

$$ P_{\mu}(Y_t = j , \tau>t) = \mu S^t (j).$$ 
In particular, 
$$P_{\mu}(\tau>t)= \mu S^t {\bf 1}=\sum_{j\in \Omega} \mu S^t(j).$$ 
Furthermore, 
$$ \lim_{t\to\infty} P_{\mu} (\tau>t)^{1/t}$$ 
exists. Moreover, if one considers the restriction of $S$ to the states accessible from the support of $\mu$, then the limit above  coincides with spectral radius of the resulting principal submatrix. 

A probability  distribution $\nu$  on $\bar \Omega$  is called a quasi-stationary distribution (QSD) if 
\begin{equation} 
\label{eq:QSD} P_{\nu} (Y_t =j | \tau > t) = \nu(j), ~j \in \Omega.
\end{equation} 
for all $t\in \Z_+$. Clearly, $\nu$ is then supported on $\Omega$, and therefore will be viewed  as a probability measure on $\Omega$.  Furthermore,  \eqref{eq:QSD} can be rewritten as 
$$\frac{  \nu S^t e_j }{\nu S^t {\bf 1}} = \frac{ \nu S^t (j) }{C_{\nu}(t) }=\nu (j),$$ 
where $e_j(k)= 1$ if $k=j$ and $0$ otherwise,  and  $C_\nu (t) = \sum_{k\in \Omega} \nu S^t(k) $. Equivalently, 
$$ \nu S^t  = C_\nu(t)  \nu.$$ 
Plugging $t=1$ into the above equation leads to the following well-known result. 
\begin{prop}
\label{prop:QSD}\leavevmode
\begin{enumerate}
\item A probability vector $\nu$ on $\Omega$  is a QSD if and only if $\nu$ is a left eigenvector for $S$ with a strictly positive eigenvalue $\lambda$. That is, if
$$ \nu S = \lambda \nu,$$
with $\lambda$ being the spectral radius of $S$ restricted to the linear space spanned by the indicators of the  support of $\nu$.
\item If $\nu$ is a QSD, then the distribution of $\tau$ under $P_\nu$ is geometric with parameter $1-\lambda$.
\end{enumerate} 
\end{prop} 

As for existence and uniqueness of the QSD, as well as for convergence of the conditioned chain to the QSD, the Perron-Frobenius theorem  \cite[Theorem 8.4.4]{horn_johnson} and the limit theorem for primitive matrices \cite[Theorems 8.5.1 and 8.5.3]{horn_johnson} yield the following result.
\begin{thm} \leavevmode
\label{thm:perron} 
\begin{enumerate} 
\item If $S$ is irreducible then there exists a unique QSD. 
\item If  $S$ is irreducible and aperiodic (also known as primitive), then for any initial distribution $\mu$,  
$$\lim_{t\to\infty} P_\mu (Y_t \in \cdot ~| \tau>t) = \nu,$$ 
 where $\nu$ is the unique QSD. 
\end{enumerate} 
\end{thm}
\section{Voter Model on a Finite Graph} \label{sec:voter_model}
Let $G=(V,E)$ be a finite, connected graph with vertex set $V$ and edge set $E$.  A {\it coloring} of $G$ is a function $f$ from the $V$ to the nonnegative integers.  The number assigned to vertex $v\in V$ by $f$, namely $f(v)$, is the color of $v$ or the opinion of $v$. We will use the terms ``color'' and ``opinion'' interchangeably.  The model we will study in this paper is a discrete-time version of the voter model on $G$. This  is a Markov chain on the colorings of $G$ which evolves during each unit of time as follows. A vertex is picked uniformly, this vertex samples a neighbor uniformly, then the former vertex adopts the opinion of the latter. We say that two neighboring vertices are in agreement if their colors are identical. Otherwise, they are in disagreement. States with no disagreements are {\it consensus states}.  Note that the set of absorbing states is the set of consensus states.

Since  $G$ is connected, the model reaches a consensus with probability $1$. This is because by construction there is positive probability to reach a consensus from any state within a certain number of steps. Since the graph is finite, this implies that the time to reach a consensus is dominated by a geometric random variable. 

We write ${\bf \eta}=(\eta_t:t\in \Z_+)$ for the discrete-time voter model on $G$, with $\eta_t (v)$ representing the color (or opinion) of vertex $v$ at time $t$. A state of the system is  therefore a coloring $\eta: V \to \Z_+$ of $G$.  The probability of a transition from $\eta$ to $\eta'$ is positive if and only if there exists  $(v,v')\in V \times V$  such that
\begin{enumerate} 
\item $\{v,v'\}\in E$;
\item $\eta'(v) = \eta (v')$; and 
\item $\eta'(u) = \eta(u)$ for all $u\neq v$.
\end{enumerate} 
Now if the pair $\eta$ and $\eta'$ satisfy the above conditions, then a transition is obtained by first uniformly  sampling the vertex $v$ among all those for which a matching $v'$ exists, and then adopting the opinion of $v'$. This leads to the following transition function: 

\begin{equation} 
\label{eq:TF_voter} p(\eta,\eta') = \frac{1}{|V|} \sum_{v\in V}\sum_{v'\sim v} \frac{ {\bf 1}_{\eta'(v)}\big(\eta(v')\big)}{\mathrm{deg}(v)}\prod_{u\ne v} {\bf 1}_{\eta'(u)}\big(\eta(u)\big).
\end{equation} 
All other transitions are not allowed.

Bringing the discussion from Section \ref{sec:QSD} into the context of the voter model on $G$, the absorption time $\tau$ is the  time of first consensus, that is 
\begin{equation} 
\label{eq:tau_defn}
\tau = \inf\{t\in \Z_+: \eta_t (v) = \eta_t(v')\mbox{ for all }v,v'\in V\}.
\end{equation} 
Using the subscript $\mathrm{V}$ to designate voter model, we write $\lambda_\mathrm{V}(G,\mu)$ for the spectral radius associated with the initial distribution $\mu$, 
\begin{equation}
\label{eq:lambda_Vmu}
\lambda_\mathrm{V}(G,\mu) = \lim_{t\to\infty} P_\mu(\tau>t)^{1/t}.
\end{equation} 

\section{Coalescing Random Walks and Time to Absorption}\label{sec:duality}
The key to our analysis is based on the duality between the continuous-time voter model and coalescing random walks. This duality is well known and the reader is directed to \cite{Durrett_particles,crw,RIO_CRW} and references therein for an exposition. In this section, we develop a discrete-time analogue of this duality. For the purposes of presentation and in order to make it useful for future work, we will consider the voter model on a general finite connected graph $G=(V,E)$. As far as the authors know, such a general treatment of this duality in discrete-time hasn't appeared in the literature before.

The first step towards finding a QSD is identifying $\lambda_\mathrm{V} (G,\mu)$. This is a nontrivial problem in general. The connection with coalescing random walks that will be described in this section simplifies the analysis of the time until consensus by  identifying the distribution of $\tau$ with the distribution of the time until two random walks on the graph first meet.  The idea is to  describe the ``flow'' or propagation of opinions back in time, tracing whose opinion each vertex inherited from previous steps, going all the way back to time zero.  Following the origin of an opinion of a given vertex  backward in time is a random walk on the graph, and the family of resulting random walks, indexed by the vertices of the graph, is a process known as coalescing random walks.  

In passing from time  $t-1$ to $t$ in the voter model, we first uniformly select a vertex $v$, then uniformly select a neighbor $u$ and assign $\eta_{t}(v)=\eta_{t-1}(u)$. For each $t\in \{1,2,\dots\}$, the sampling of vertex $v$ and its neighbor is independent of and identically distributed as the respective sampling for other times. Furthermore, this sampling is also independent of the actual opinions up to time $t-1$. Fix some time $T\in\N$.  We will construct a random directed graph $\mathcal{G}_T$ on $V\times \{0,\dots,T\}$  which would represent the same process, but with time reversed. Why reverse time? Because we eventually  want to trace whose original opinion (opinion at time $0$) each of the vertices holds at time $T$. Below we denote a directed arrow from $(u,n)\in V \times \{0,\dots,T\}$ to $(u',n')\in V \times \{0,\dots,T\}$ by $(u,n)\to(u',n')$.

We now describe the construction. This is done in three steps. 

\subsubsection*{1. Adopting others' opinions.}
 If at time $t=T-n$ the vertex $v$ is selected to adopt the opinion of vertex  $u$ at time $t-1=T-(n+1)$, we will draw a directed arrow from $(v,n)$ to $(u,n+1)$.  We begin from $t=T$, and end at $t=1$.  This describes which vertex got whose opinion and when. See Figure \ref{fig:adopting_others} for an illustration of this procedure on a star graph with $T=6$. Note that the $t$-time of the voter model runs from left to right while the $n$-time of the random directed graph runs from right to left.

\subsubsection*{2. Keeping one's opinion.}
Since all vertices but one keep their opinions from time $t-1$ to time $t$, we add arrows to represent this as well. To do that let $h(n)$ denote the unique vertex $(v,n)$ with an arrow to some $(u,n+1)$, as obtained in Step 1.  For all $v\in V-\{ h(n)\}$, we draw an arrow from $(v,n)$ to $(v,n+1)$. See Figure \ref{fig:not_adopting} for an illustration of this stage. 

\subsubsection*{3. Removing useless arrows.}
At the end of Step 2, for every $v\in V$ there exists a unique path  from $(v,0)$ to  $(\cdot,T)$. That is,  given $v\in V$, the unique path is a sequence $(v_0=v,0)\to (v_1,1)\to \dots (v_T,T)$ where $v_{n+1}$ is the unique vertex $v\in V$  satisfying $(v_n,n)\to (v,n+1)$. Since the path is determined by the choice of $v$ and $T$,  we denote it by ${\bf X}^T(v)=(X_n^T(v):n=0,\dots,T)$, where $X_n^T(v)=v_n$. An arrow is useless if no path  $(v,0)\to \dots \to (v_T,T)$ uses it. We will remove those from our graph $\mathcal{G}_T$.  With this our construction is complete. This stage is illustrated in Figure \ref{fig:final}.

\begin{figure}
\centering
\includegraphics[width=\linewidth]{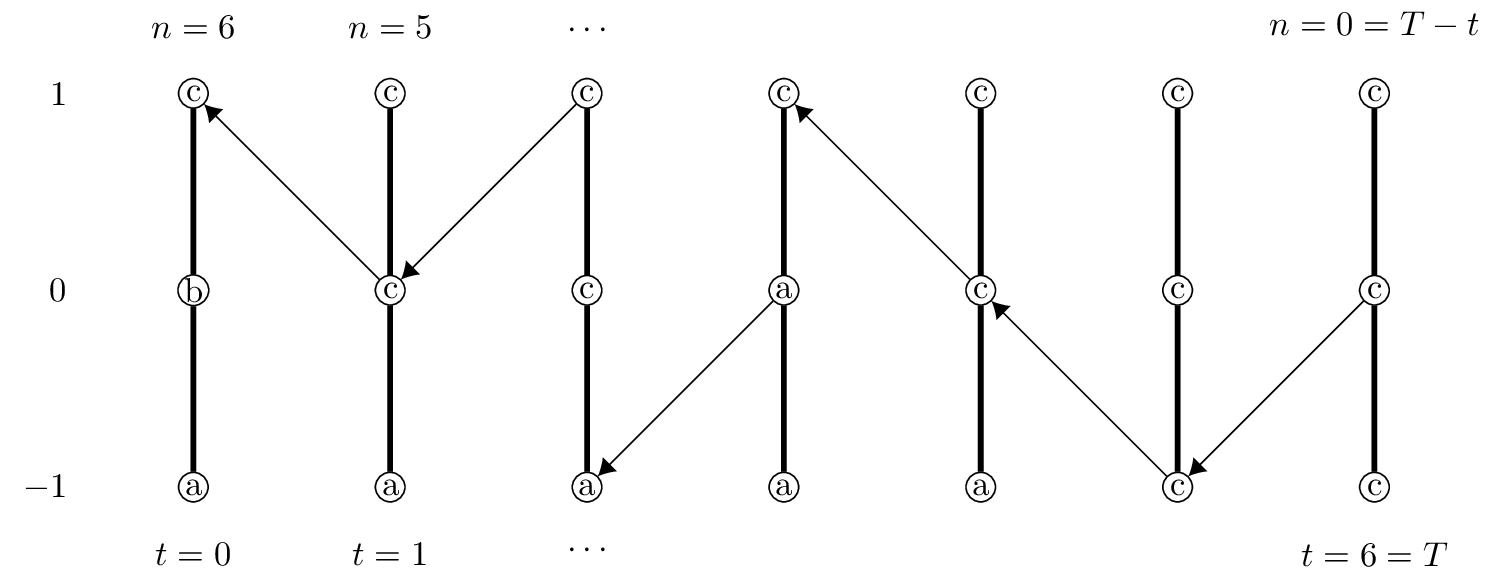}
\caption{Realization of the voter model on a simple star graph as a function of time.  Vertices are the circles,  labeled  $0,\pm 1$ and edges are the vertical line segments. Opinion at each  vertex is the letter inscribed in the circle. The time for the voter model appears at the bottom, while  reversed time, the second component in the vertices of the random graph $\mathcal{G}_T$, appears at the top. An arrow  from  $(v,n)$   to vertex $(u,n+1)$  represents vertex $v$ adopting at time $t=T-n$ the opinion of vertex $u$ at time $t-1=T-(n+1)$.  } \label{fig:adopting_others}
\end{figure}

\begin{figure} 
\centering
\includegraphics[width=.95\linewidth]{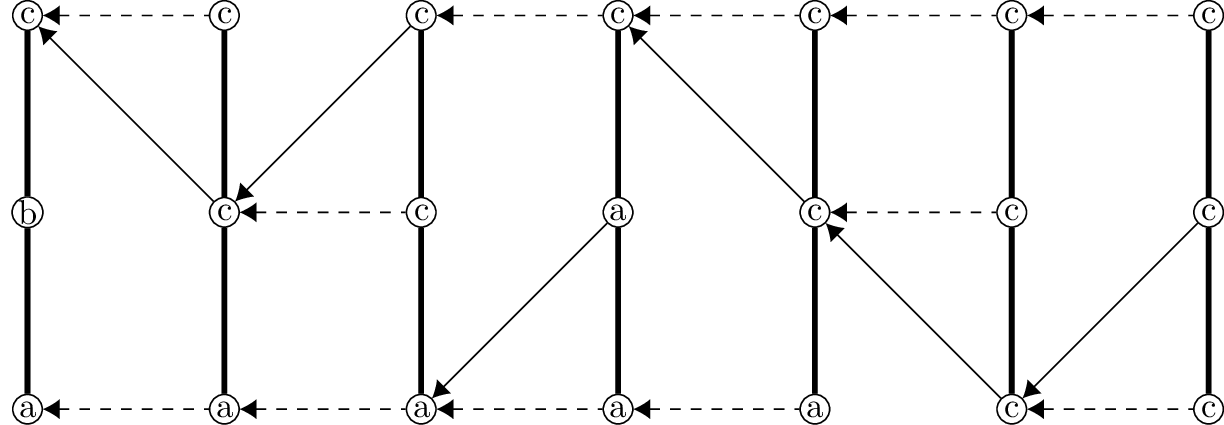}
\caption{Same realization as Figure \ref{fig:adopting_others}, with dashed horizontal arrows representing vertices keeping their opinion.} 
\label{fig:not_adopting}
\end{figure}

\begin{figure}
\centering
\includegraphics[width=.95\linewidth]{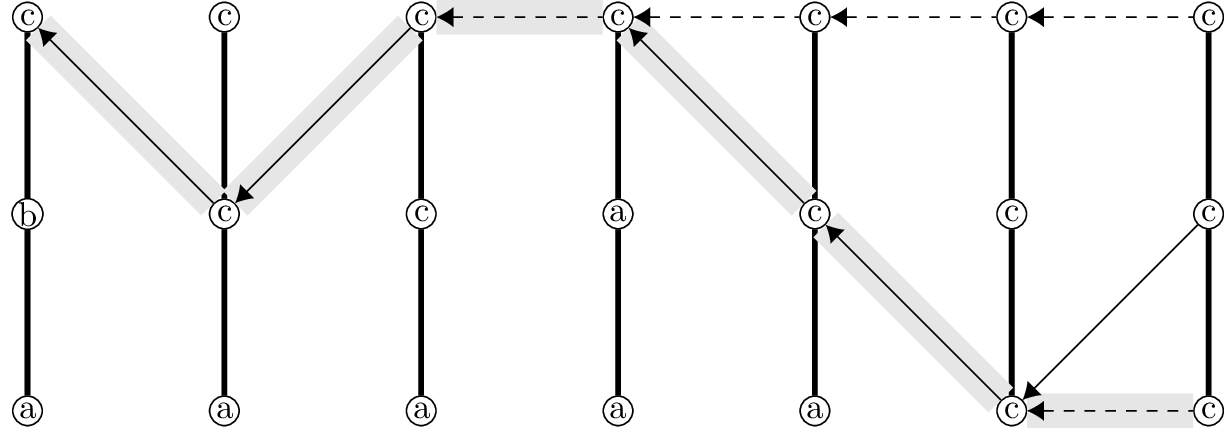}
\caption{The random graph $\mathcal{G}_6$, obtained  after removing useless arrows from Figure \ref{fig:not_adopting}. The path of the random walk ${\bf X}^6 (-1)$ has a shadow.}
\label{fig:final} 
\end{figure}

Now that we have completed the construction of the random graphs, we do some analysis.  Each of the random graphs is determined by Step 1 while Steps 2 and 3 are deterministic functions of it. From our construction, for each $v\in V$,  ${\bf X}^T(v)$ is a Markov chain on $G$ whose initial distribution is the point-mass at $v$ and with transition function 
$$p(u,u) =  \frac{|V|-1}{|V|},~ p(u,w) = \frac{1}{|V|}\frac{1}{\mathrm{deg}(u)},~ \{u,w\}\in E.$$
Equivalently, ${\bf X}^T(v)$ is a lazy random walk on $G$ which starts at $v$ and has probability $1-\frac{1}{|V|}$ of staying put at each step. 

Next we define a system of coupled random walks ${\bf Y}= (Y(v):v \in V)$ whose distribution up to time $T$ coincides with that of ${\bf X}^T$. We begin by setting $Y_0(v)=v$ for each $v\in V$. Assuming that $Y_s(\cdot)$ is defined for $s\le t$, we uniformly and independently sample ${\bf v} \in V$ and a neighbor ${\bf u}$ of ${\bf v}$. If $Y_t (\cdot)\ne {\bf v}$, then we set $Y_{t+1}(\cdot)= Y_t (\cdot)$. Otherwise we set $Y_{t+1} (\cdot) = {\bf u}$. We will refer to ${\bf Y}$ as the {\it coalescing random walks}. Note that the joint distribution of the walks ${\bf X}^T(v),v \in V$ coincides with the restriction of the coalescing random walks to the time interval $\{0,\dots,T\}$. 

For any two distinct vertices $v,v'\in V$, let 
$$\sigma_{v,v'}^T = \inf\{n\in \{0,\dots,T\}: X_n^T(v)=X_n^T(v')\},\mbox{ with }\inf \emptyset = \infty, $$
and let 
$$\sigma_{v,v'}= \inf \{n \in \Z_+: Y_n (v)=Y_n (v')\}.$$ 
Then
\begin{equation} 
\label{eq:use_walks}P( \sigma_{v,v'}^T = t) = P(\sigma_{v,v'}=t),~t\le T.
\end{equation}
Also, let 
$$\sigma = \max_{v,v'}\sigma_{v,v'}.$$ 

Continuing, assume that initially all opinions are distinct. Without loss of generality this can be expressed as $\eta_0(v)=v,~v\in V$. We denote this initial state of the system by ${\bf i}$. Observe the following: 
\begin{enumerate} 
\item Under $P_{\bf i}$, the distribution of $\tau$, the consensus time, and of $\sigma$ coincide. 
\item Let $\mu$ be any initial distribution for the voter model on $G$. Then the distribution of  $\tau$ under $P_\mu$ is stochastically dominated by its distribution under $P_{\bf i}$. 
\end{enumerate} 
Before we continue, we would like to recall a fundamental but useful fact. Suppose that $\eta$ is a $\N$-valued random variable with the property that for some $\lambda \in [0,1]$, 
$$\lambda = \lim_{n\to\infty} P(\eta>n)^{1/n}.$$ 
Letting $\rho\geq 1$, we can write 
\begin{equation} 
\label{eq:series} E[\rho ^\eta ] =\sum_{k=1}^\infty \rho^k \big( P( \eta > k-1) - P(\eta>k)\big) = 
\rho +(\rho-1) \sum_{k=1}^\infty\rho^k  P( \eta>k),
\end{equation} 
and therefore it follows from the Cauchy-Hadamard theorem that the radius of convergence of the power series on the right-hand side of \eqref{eq:series} is $1/\lambda$. In particular,  
\begin{equation} \label{eq:radius} \frac{1}{\lambda} = \sup\{\rho : E[\rho^\eta]< \infty\}.
\end{equation} 

Now define 
$$\lambda_\mathrm{CRW}(G)=\lim_{n\to\infty} P(\sigma>n)^{1/n}$$
and recall that 
$$ \lambda_\mathrm{V}(G,\mu) = \lim_{n\to\infty}P_{\mu}(\tau>n)^{1/n}.$$ 
Both limits exists as $\sigma$ and $\tau$ are hitting times of finite-state Markov chains and decay geometrically (possibly with a polynomial correction). Therefore it follows from \eqref{eq:radius} that 
\begin{equation}\label{eq:CRW_radius}
\frac{1}{\lambda_\mathrm{CRW}(G)} = \sup\{\rho:  E[\rho^{\sigma}]<\infty\},
\end{equation}
as well as 
\begin{equation}\label{eq:variational}
\frac{1}{\lambda_\mathrm{V}(G,\mu)} = \sup\{\rho:E_\mu [\rho^\tau]<\infty\}.
\end{equation}
The observations above also imply 
\begin{equation}
    \label{eq:those_lambdas}
     \lambda_\mathrm{V}(G,\mu)\le  \lambda_\mathrm{CRW}(G),
\end{equation}
and  $\mu ={\bf i}$ is a sufficient condition for equality. 

We will now examine other sufficient conditions for an equality. For  $\rho \ge 1$,
\begin{equation} 
\label{eq:equivalence}
\max_{v,v'} E[\rho^{ \sigma_{v,v'}}]\le E[\rho^{\sigma}]\le \sum_{v,v'}E[\rho^{\sigma_{v,v'}}].
\end{equation} 
It therefore follows that $E[\rho^{ \sigma}]< \infty$ if and only if $\max_{v,v'}E[\rho ^{\sigma_{v,v'}}]<\infty$. Now we can write 
\begin{align}
\lambda_\mathrm{CRW}(G) &=1/ \sup\{\rho:  \max_{v,v'}E[\rho ^{\sigma_{v,v'}}]<\infty\}\nonumber \\
&=\max_{v,v'}1/ \sup\{\rho: E[\rho ^{\sigma_{v,v'}}]<\infty\}\nonumber \\
& =  \max_{v,v'} \lim_{n\to\infty} P(\sigma_{v,v'}>n)^{1/n}\nonumber \\ 
&=\lim_{n\to\infty}\left(\max_{v,v'} P(\sigma_{v,v'}>n)\right)^{1/n}\label{eq:maximizer}
\end{align}
where the first equality follows from \eqref{eq:CRW_radius} and \eqref{eq:equivalence} and the third equality uses \eqref{eq:radius}.
 
Fix $t\ge 0$ and let $v,v'\in V$ be distinct. If $X^t_t(v)=u\ne u'= X^t_t(v')$ and $\eta_0(u)\ne \eta_0(u')$, then necessarily $\eta_t(v)=\eta_0(u)$ and $\eta_t(v')=\eta_0(u')$, hence $\tau >t$. Summing, we have 
$$P_\mu (\tau>t) \ge \sum_{u,u'}P_\mu\left(X_t^t(v)=u,X_t^t(v')=u',\eta_0(u)\ne \eta_0(u')\right),$$
and we can sum over all $u,u'$ because $\eta_0(u)\ne \eta_0(u')$ implies $u\ne u'$. Since $\eta_0$ is independent of the random walk, we can decouple the condition on the random walk from the condition on the initial opinions, and limit the summation only to pairs $u,u'$ where $u\ne u'$. Using \eqref{eq:use_walks}, this gives 
\begin{align*} P_\mu(\tau>t) &\ge  \sum_{u\ne u'}P\left(X_t^t(v)=u,X^t_t(v')=u'\right)P_{\mu}\big(\eta_0(u)\ne \eta_0(u')\big)\\
 & \ge \sum_{u\ne u'}P\left(X_t^t(v)=u,X^t_t(v')=u'\right) c\\
 & = P(\sigma_{v,v'}>t)\,c
\end{align*} 
where 
\begin{equation}\label{eq:distinct}
\begin{split}
c&=\min_{u\ne u'}P_{\mu}\big(\eta_0(u)\ne \eta_0(u')\big)\\
&\geq P_\mu(\text{all initial opinions are distinct}).
\end{split}
\end{equation}
As long as $c>0$, this implies that the geometric decay  of $\tau$ starting from $\mu$ is at least as slow as that of $\sigma_{v,v'}$. Since $v,v'$ were arbitrary, it follows from \eqref{eq:maximizer} that 
\begin{equation}\label{eq:relaxed_1}
\lambda_\mathrm{CRW}(G) \le \lambda_\mathrm{V}(G,\mu).
\end{equation}
In view of \eqref{eq:those_lambdas} and \eqref{eq:distinct}, we have established an equality in the case where with positive probability, all initial opinions are distinct. 

Next, we relax the  condition for equality a little further. Suppose that $\mu_0$ is an initial distribution  on any number of opinions such that for some $t_0$, $P_{\mu_0} (\eta_{t_0}(u)\ne \eta_{t_0}(u'))>0$ for all $u\neq u'$. Denote the distribution of $\eta_{t_0}$ by $\mu$ and note that $\lambda_\mathrm{V}(G,\mu)=\lambda_\mathrm{CRW}(G)$ follows from \eqref{eq:those_lambdas}, \eqref{eq:distinct}, and \eqref{eq:relaxed_1}. 
For any $\rho \ge 1$, we can use the Markov property to write
\begin{align*}
      E_{\mu_0} [\rho^{\tau} ] &\ge  E_{\mu_0} [\rho^{\tau},\tau>t_0] \\
      & = E_{\mu_0} \left[\left(1-{\bf 1}_{\{\tau\leq t_0\}}\right)\rho^{t_0}  E_{\eta_{t_0}}\left[\rho^{\tau}\right]\right]\\
       & = \rho^{t_0}\Big(E_{\mu_0} \left[E_{\eta_{t_0}}\left[\rho^{\tau}\right]\right]- E_{\mu_0}\left[{\bf 1}_{\{\tau\leq t_0\}}E_{\eta_{t_0}}\left[\rho^\tau\right]\right]\Big)\\
       & \ge  \rho^{t_0}\left(E_{\mu}\left[\rho^ \tau\right]-1\right).
\end{align*}
Hence if $E_{\mu}\left[\rho^ \tau\right]$ is infinite, then so is $E_{\mu_0}\left[\rho^ \tau\right]$. Therefore, it follows from \eqref{eq:variational} that $\lambda_\mathrm{V}(G,\mu_0) \ge \lambda_\mathrm{V}(G,\mu)$. Thus $\lambda_\mathrm{V}(G,\mu_0)=\lambda_\mathrm{CRW}(G)$ and we have proved the next proposition.
\begin{prop}
\label{prop:finally}
Let $G=(V,E)$ be a finite connected graph. 
\begin{enumerate} 
\item For any initial opinion distribution $\mu$
$$ \lambda_\mathrm{V}(G,\mu) \le \lambda_\mathrm{CRW}(G).$$ 
\item An equality holds in each of the following cases: 
\begin{enumerate} 
\item With positive probability, all initial opinions are distinct. 
\item There exists $t\ge 0$, such that for every distinct $u,u'\in V$, $P_{\mu}(\eta_t(u)\ne \eta_t(u'))>0$. 
\end{enumerate} 
\end{enumerate}
\end{prop}

An important observation is that we can always attain an equality with only two opinions. 
\begin{cor}
Let $G=(V,E)$ be a finite connected graph and suppose that $\mu$ is the initial distribution on the opinions $\{0,1\}$ where all vertices have opinion $0$ except for one uniformly chosen vertex which has opinion $1$. Then 
$$\lambda_\mathrm{V}(G,\mu) = \lambda_\mathrm{CRW}(G).$$ 
\end{cor} 

\begin{xmpl}\label{ex:complete_graph}
This case was treated in \cite[Section 5]{dickman} for the continuous-time analogue. 
Let $K_n$ be the complete graph with $n$ vertices and consider the voter model on $K_n$ for $n\ge 3$ with two opinions: ``yes'' and ``no''. This Markov chain, when restricted to the nonabsorbing states, is irreducible and aperiodic. Hence it follows from Theorem \ref{thm:perron} that the chain conditioned on nonabsorption converges to the unique QSD. Writing down the eigenvalue equation from Proposition \ref{prop:QSD}, we have 
\begin{align*} \lambda\, \nu_n (k) &=  \frac{k(k-1)+(n-k)(n-k-1)}{n(n-1)}\nu_n (k)\\
&\quad + \frac{(k+1)(n-k-1)}{n(n-1)}\nu_n (k+1) + \frac{(n-k+1)(k-1)}{n(n-1)}\nu_n (k-1),
\end{align*} 
where $k$ represents the number of ``yes'' opinions, and $\nu_n(0)=\nu_n(n)=0$. 
Letting  $A=k$ and $B=n-k$, the numerators on the right-hand side are equal to $A^2 -A +B^2 -B$, $AB+B-A-1$, $AB+A-B-1$, and adding them up gives us $A^2+B^2 +2AB-(A+B)-2=n(n-1)-2$. Therefore, the equation is solved by choosing  $\nu_n(k) = \frac{1}{n-1}$ and $\lambda= \frac{n(n-1)-2}{n(n-1)} = 1- \frac{2}{n(n-1)}$. Proposition \ref{prop:finally}-2b gives  $\lambda_\mathrm{V}(K_n,\mu)=\lambda_\mathrm{CRW}(K_n)$, whenever $\mu$ is not supported on consensus states, due to the above claimed  irreducibility. Given two  walkers at different vertices, the probability that they meet in the next step is $\frac{2}{n(n-1)}$, so it follows from  \eqref{eq:maximizer} that $\lambda_\mathrm{CRW}(K_n) = 1 - \frac{2}{n(n-1)}$, as established above through a direct calculation of the QSD.
\end{xmpl} 
Summarizing, the QSD for the voter model on the complete graph $K_n$ is uniform on the nonabsorbing states. In particular, the QSDs do not converge to a probability distribution as $n\to\infty$. As we will see in Section \ref{sec:main}, the situation is more interesting for the complete bipartite graphs $K_{n,m}$. 

\section{Voter model on Complete Bipartite Graphs}\label{sec:main}
Let $K_{n,m}=(V,E)$ be the bipartite graph whose vertex $V$ set is the disjoint union of $L$ and $S$, where $|L|=n$, $|S|=m$, $m\le n$, and its edge set is $E=\{\{l,s\}:l\in L, s\in S\}$. Though these graphs are quite simple, an interesting structure appears when considering the behavior of the voter model conditioned on not reaching consensus for a long time.

We will study QSDs for the voter model on $K_{n,m}$ with two opinions, ``0'' and ``1'', also referred to as ``no'' and ``yes'', respectively. As noted before, we assume $m\le n$. We will also impose the following additional constraints which we need in order to guarantee irreducibility:  
\begin{equation} 
\label{eq:bipart_assum}
\begin{split} 
m&\ge 2 \mbox{ or }\\
m=1 &\mbox{ and } n\ge 3. 
\end{split} 
\end{equation}

The set $\Delta$ of absorbing states for the voter model on $K_{n,m}$ is given by 
$$\Delta = \{\eta: \eta\equiv 0\mbox{ or } \eta \equiv 1\}.$$ 
In addition, the states in 
$$BP:=\{\eta:\eta(l)=1-\eta(s),l \in L,s\in S\}$$ 
are not accessible from any state not in $BP$. We will therefore eliminate the subsets $\Delta$ and $BP$ from our state space for the model. A routine argument shows that under \eqref{eq:bipart_assum}, the $2$-opinion voter model is now irreducible and aperiodic. Hence it follows from Theorem \ref{thm:perron} and Proposition \ref{prop:finally} that starting from any initial distribution $\mu$ supported on $(\Delta \cup BP)^c$ and conditioning on not reaching consensus, the model converges to the unique QSD which is also supported on $(\Delta \cup BP)^c$. We denote this QSD by $\pi_{n,m}$ and note that it is a left eigenvector corresponding to the eigenvalue $\lambda_{n,m} =\lambda_\mathrm{V}(K_{n,m},\mu)= \lambda_\mathrm{CRW}(K_{n,m})$ for the restriction of the transition function of the voter model to $(\Delta \cup BP)^c$.

To continue our analysis, we will exploit the symmetry among vertices within each group. Instead of following the opinion on each vertex, we will follow the number of ``yes'' opinions in each of the groups $S$ and $L$. This leads to a Markov chain on the state space  $\{0,\dots,n\}\times\{0,\dots,m\}$. Each state is an ordered pair $(k,h)$, with $k$ representing the number of ``yes'' in group $L$ and $h$ representing the number of ``yes'' in $S$. Observe that the only allowed transitions are the following: 
\begin{enumerate}
    \item $(k,h) \to (k+1,h)$.   This happens if a ``no'' vertex in $L$ is sampled and adopts a ``yes'' from $S$. The probability of such a transition is therefore 
    $\frac{n-k}{n+m}\frac{h}{m}$. 
     \item $(k,h)\to(k-1,h)$.  This happens if a ``yes'' vertex in $L$ is sampled and adopts a ``no'' from $S$. The probability is therefore  $\frac{k}{n+m}\frac{m-h}{m}$.
\item $(k,h)\to(k,h+1)$. This is item 1. above with the roles of $L$ and $S$ interchanged and hence occurs with probability 
    $\frac{m-h}{n+m}\frac{k}{n}$. 
    \item $(k,h)\to(k,h-1)$. This is item 2. above with the roles of $L$ and $S$ interchanged. Similarly, this happens with probability 
    $\frac{h}{n+m}\frac{n-k}{n}$. 
    \item $(k,h)\to(k,h)$. This happens with probability  $\frac{k}{n+m}\frac{h}{m}+\frac{h}{n+m}\frac{k}{n}+\frac{n-k}{n+m}\frac{m-h}{m}+\frac{m-h}{n+m}\frac{n-k}{n}=\frac{kh+(n-k)(m-h)}{nm}$.
\end{enumerate}
Of course,  $(0,0),(n,m)$ are the unique  absorbing states, and the set $BP$ collapses into two states,  $(0,m)$ and $(n,0)$, not accessible from any other state. Thus eliminating these four states, the chain is irreducible. As a result, it possesses a unique QSD which we denote by $\mu_{n,m}$. Recall that $\lambda_{n,m} = \lambda_\mathrm{CRW}(K_{n,m})$ coincides with the geometric rate of the time to absorption in the voter model. The absorption time for the new chain  from any initial state coincides with the time to absorption for the voter model starting from any state with matching numbers of opinions in both $S$ and $L$, therefore, it follows from Proposition \ref{prop:QSD} that the eigenvalue corresponding to $\mu_{n,m}$ is equal to $\lambda_{n,m}$. Now fix any state $(k,h)$ for our new chain.   By looking at all possible transitions we obtain the following equation for $\mu_{n,m}$: 
\begin{equation} 
\label{eq:stationary_bipart}
\begin{split}
    \lambda_{n,m} \,\mu_{n,m}(k,h)&=\mu_{n,m}(k,h)\frac{kh+(n-k)(m-h)}{nm}\\&+\mu_{n,m}(k-1,h)\frac{(n-k+1)h}{m(n+m)}+\mu_{n,m}(k+1,h)\frac{(k+1)(m-h)}{m(n+m)}\\
    &+\mu_{n,m}(k,h-1)\frac{(m-h+1)k}{n(n+m)}+\mu_{n,m}(k,h+1)\frac{(h+1)(n-k)}{n(n+m)}.
\end{split}
\end{equation}

In order to extract more information, we first compute $\lambda_{n,m}$. 
\begin{prop}\label{pr:lambda_bipart}
 \begin{equation*} 
\begin{split}
\lambda_{n,m}=\lambda_\mathrm{CRW}(K_{n,m}) &= 
   1 - \frac{2}{n+m}\left(1-\sqrt{1-\frac{1}{2n}-\frac{1}{2m}}\right)\\
   &=1-\frac{\gamma_{n,m}}{n+m},
\end{split}
\end{equation*} 
where $$ \gamma_{n,m} = 2 \left(1-\sqrt{1-\frac{1}{2n}-\frac{1}{2m}}\right).$$
\end{prop}

\begin{proof} 
We assume first $m>1$. From \eqref{eq:maximizer} it is enough to consider only two coalescing random walks on $K_{n,m}$. The two CRW paths can  be in either one of the following states: both walks are in different vertices of  $L$, both walks are in different vertices of  $S$, one walk is in $S$ and another in $L$, or they are both at the same vertex.  Label these four states of the system as $1,2,3,4$,  respectively. Of course, $4$ is the absorbing state for the CRW, so we will omit it from our calculations. From each of the states $1,2,3$, the system will stay put with   probability $1-\frac{2}{n+m}$. From state $1$ the system can transition to  $3$ with probability $\frac{2}{n+m}$, and  similarly,  from state $2$ the system can  transition to state $3$ with probability $\frac{2}{n+m}$. Finally, from state $3$, the system can transition to $1$ or to $2$ with respective probabilities $\frac{1}{n+m}\frac{n-1}{n}$ and $\frac{1}{n+m}\frac{m-1}{m}$. As a result the substochastic transition function on states $1,2,3$ is 
\begin{equation} 
\label{eq:walk_matrix}
\begin{pmatrix}
\frac{n+m-2}{n+m} & 0 & \frac{2}{n+m}\\ 
0& \frac{n+m-2}{n+m} & \frac{2}{n+m}\\
\frac{1}{n+m}\frac{n-1}{n} & \frac{1}{n+m}\frac{m-1}{m} & \frac{n+m-2}{n+m}
\end{pmatrix}.
\end{equation}
Since from both states $1$ and $2$ the transitions are either to themselves with the same probability, or to $3$ with the complementary probability, we can consolidate these states into one, leading to the matrix 
\begin{equation}
\label{eq:reduced_matrix} 
\begin{pmatrix}
 \frac{n+m-2}{n+m} & \frac{2}{n+m} \\
  \frac{1}{n+m}\left(\frac{n-1}{n}+\frac{m-1}{m}\right) & \frac{n+m-2}{n+m}
\end{pmatrix} .
\end{equation} 
The characteristic equation is 
$$\left(\lambda - \frac{n+m-2}{n+m}\right)^2- \frac{2}{(n+m)^2}\frac{2mn-m-n}{nm}=0.$$
Therefore the two eigenvalues, $\lambda_+$ and $\lambda_-$, are given by 
\begin{align*}  \lambda_\pm &= 1-\frac{2}{n+m}\pm \frac{1}{n+m}\sqrt{4-\frac{2}{n}-\frac{2}{m}}\\
& = 1 - \frac{2}{n+m} \left ( 1 \pm  \sqrt{1 - \frac{1}{2n}-\frac{1}{2m}}\right),
\end{align*} 
and the largest eigenvalue is obtained by choosing $\lambda_{-}$ (using the ``$-$'' sign),  giving the expression in the statement. 

It remains to consider the case $m=1$. In this case, state $2$ is not possible. We therefore eliminate the second row and the second column from \eqref{eq:walk_matrix}, ending up with the matrix \eqref{eq:reduced_matrix} and then continue as before. 
\end{proof} 

The next result gives a more direct connection between $\lambda_{n,m}$ and $\mu_{n,m}$. 
\begin{prop}
\label{pr:lambda_mu_connection}
\begin{equation*} 
\lambda_{n,m}= 1- \frac{2}{n+m}\big(\mu_{n,m}(1,0) +  \mu_{n,m}(0,1)\big)
\end{equation*} 
\end{prop}
\begin{proof}
Let $(K,H)$ be a random vector representing the number of ``yes'' in $L$ and $S$, respectively, whose distribution is $\mu_{n,m}$. Recalling that $\mu_{n,m}(0,0)=\mu_{n,m}(n,m)=0$, we can sum both sides of \eqref{eq:stationary_bipart} over $-1\leq k\leq n+1$ and $-1\leq h\leq m+1$ while eliminating from the sum the pairs $(k,h)=(0,0)$ and $(k,h)=(n,m)$ to obtain 
\begin{align*}
\lambda_{n,m} =&\frac{1}{nm} \big(E[KH]+E[(n-K)(m-H)] \big)\\
& + \frac{1}{m(n+m)} \big(E[(n-K)H]-\mu_{n,m}(n-1,m)m\big)\\
&+ \frac{1}{m(n+m)}\big(E[K(m-H)]-\mu_{n,m}(1,0)m\big)\\ 
& + \frac{1}{n(n+m)}\big(E[(m-H) K]-\mu_{n,m}(n,m-1)n\big)\\ 
&+ \frac{1}{n(n+m)}\big(E[H(n-K)]-\mu_{n,m}(0,1)n\big)\\
=&1-\frac{1}{n+m}\big(\mu_{n,m}(n-1,m)+\mu_{n,m}(1,0)+\mu_{n,m}(n,m-1)+\mu_{n,m}(0,1)\big)
\\
= &1- \frac{2}{n+m}\big(\mu_{n,m}(1,0) +  \mu_{n,m}(0,1)\big)
\end{align*}
where the last equality follows from invariance under relabeling of the two opinions.
\end{proof}

We now show how to use this proposition to compute $\mu_{n,1}(j,0)$. As the state $(0,1)$ is in $BP$ and not in the support of $\mu_{n,1}$, using Proposition \ref{pr:lambda_mu_connection} along with Proposition \ref{pr:lambda_bipart} gives 
\begin{equation}\label{eq:rec_initial}
\begin{split}
\mu_{n,1} (1,0)=\frac{n+1}{2} (1-\lambda_{n,1})&= 1- \sqrt{\frac{1}{2}-\frac{1}{2n}}\\
&=\frac{\gamma_{n,1}}{2}.
\end{split}
\end{equation}
Writing out \eqref{eq:stationary_bipart} for this case leaves us with 
$$\left(\lambda_{n,1}-\frac{n-k}{n}\right)\mu_{n,1}(k,0) = \mu_{n,1}(k+1,0)\frac{k+1}{n+1}+ \mu_{n,1}(k,1)\frac{n-k}{n(n+1)}.$$ 
Due to invariance under relabeling of the two opinions, we have $\mu_{n,1}(k,1) = \mu_{n,1}(n-k,0)$. In addition, from Proposition \ref{pr:lambda_bipart} we have
$$\lambda_{n,1} - \frac{n-k}{n}= \frac{k}{n}- \frac{\gamma_{n,1}}{n+1}.$$
We therefore obtain the equation 
\begin{equation} 
\label{eq:star_QSD} \left( \frac{k}{n}-\frac{\gamma_{n,1}}{n+1}\right)\mu_{n,1}(k,0)= \mu_{n,1}(k+1,0)\frac{k+1}{n+1} + \mu_{n,1}(n-k,0) \frac{n-k}{n(n+1)}.
\end{equation}
This nonlocal recurrence relation can be solved through iteration. Having calculated $\mu_{n,1}(1,0)$ in \eqref{eq:rec_initial} and recalling that $\mu_{n,1}(n,0)=0$, we can plug $k=n-1$ into \eqref{eq:star_QSD} and obtain  
$$ \left(\frac{n-1}{n} - \frac{\gamma_{n,1}}{n+1}\right)  \mu_{n,1}(n-1,0) = \mu_{n,1}(1,0)\frac{1}{n(n+1)},$$
or 
$$\mu_{n,1}(n-1,0)  = \frac{\gamma_{n,1}}{2(n^2-n \gamma_{n,1}-1)}.$$
We can repeat this procedure inductively. Having calculated $\mu_{n,1}(j,0)$ and $\mu_{n,1}(n-j,0)$ for $j=1,\dots,k<n$, we can use \eqref{eq:star_QSD} to recover $\mu_{n,1}(k+1,0)$ and then $\mu_{n,1}(n-k-1,0)$.

Next we prove a technical lemma which gives the asymptotic behavior of $\mu_{n,m}$ as $n\to\infty$. 
\begin{lem}
\label{lem:bold_center}
Suppose $k\ge 0$ and $h>0$. Then 
$$ \lim_{n\to\infty}\mu_{n,m}(k,h)=0.$$ 
\end{lem} 
\begin{proof}
Assume  $k,h>0$.  From \eqref{eq:stationary_bipart} we have
$$ \left(\lambda_{n,m} - \frac{kh+(n-k)(m-h)}{nm}\right)\mu_{n,m}(k,h) = \big(1+o(1)\big) \frac{h}{m}\mu_{n,m}(k-1,h) +O\left(\frac{1}{n}\right)$$ 
as $n\to\infty$. From Proposition \ref{pr:lambda_bipart}, we know that $\lambda_{n,m}\to 1$ as $n\to\infty$. Hence 
\begin{equation}\label{eq:asymptotic}
\mu_{n,m}(k,h)=\big(1+o(1)\big)\mu_{n,m}(k-1,h)+O\left(\frac 1n\right)\text{ as }n\to\infty.
\end{equation} 

Now suppose for the sake of a contradiction that $\limsup_{n\to\infty} \mu_{n,m}(k',h')=\epsilon>0$ for some $k'\geq 0$ and $h'>0$. Then \eqref{eq:asymptotic} implies that $\lim_{j\to\infty}\mu_{n_j,m}(k'+1,h')=\epsilon$ along some subsequence $n_1,n_2,\dots$. Similarly, \eqref{eq:asymptotic} can be used again to show that $\lim_{j\to\infty}\mu_{n_j,m}(k'+2,h')=\epsilon$ along the same subsequence. Reasoning inductively both forwards and backwards in $k$, it follows that $\lim_{j\to\infty}\mu_{n_j,m}(k,h')=\epsilon$ for all $k\geq 0$. In particular, for $j$ large enough we have $\mu_{n_j,m}(k,h')>\epsilon/2$ for $0\leq k\leq \lceil 2/\epsilon\rceil$. Hence $\sum_{k}\mu_{n_j,m}(k,h')>1$, a contradiction.
\end{proof} 

In the following proposition we calculate the pointwise limit of $\mu_{n,m}$ as the size of the large partition tends to infinity. 
\begin{prop}
\label{prop:gamma_m}
Let $f_{\gamma_m}$ be as in \eqref{eq:fgamma} with $\gamma_m$ defined in Theorem \ref{th:main}. Then for $k\geq 0$
$$\mu_{\infty,m}(k,0)  := \lim_{n\to\infty}\mu_{n,m}(k,0) =\frac 12 f_{\gamma_m}(k).$$
\end{prop}
\begin{proof}
From Lemma \ref{lem:bold_center} we know that $\lim_{n\to\infty}\mu_{n,m}(0,1)=0$. Therefore it follows from Propositions \ref{pr:lambda_bipart} and \ref{pr:lambda_mu_connection} that 
\begin{equation}
\label{eq:base_asymp}
\mu_{\infty,m}(1,0) =\frac {\gamma_m}{2}.
\end{equation}
Returning to \eqref{eq:stationary_bipart} with $k\geq 1$ and $h=0$, we can write 
$$\lambda_{n,m}\, \mu_{n,m}(k,0) = \mu_{n,m}(k,0) \frac{n-k}{n} + \mu_{n,m}(k+1,0) \frac{k+1}{n+m}+\mu_{n,m}(k,1) \frac{n-k}{n(n+m)}.$$
Rearranging this equation while using Proposition \ref{pr:lambda_bipart} again leads to
\begin{equation}\label{eq:pre_limit}
\mu_{n,m}(k+1,0)=\frac{n+m}{k+1}\left(\mu_{n,m}(k,0)\left(\frac{k}{n}-\frac{\gamma_{n,m}}{n+m}\right)-\mu_{n,m}(k,1)\frac{n-k}{n(n+m)}\right).
\end{equation}
Letting $n\to\infty$ in \eqref{eq:pre_limit} while recalling that $\gamma_{n,m}\to\gamma_m$ and $\mu_{n,m}(k,1)\to 0$ results in
$$\mu_{\infty,m}(k+1,0)=\mu_{\infty,m}(k,0)\frac{k-\gamma_m}{k+1}.$$
Now we can argue inductively starting from \eqref{eq:base_asymp} to conclude that
\begin{align*}
\mu_{\infty,m}(k,0)&=\frac{\gamma_m}{2} \frac{1}{k!} \prod_{j=1}^{k-1}(j-\gamma_m)\\
&=\frac{1}{2} f_{\gamma_m}(k),~k\in\N.
\end{align*}
\end{proof} 

Finally we can give the proof of our main result.
\begin{proof}[Proof of Theorem \ref{th:main}]
By the invariance under relabeling of the two opinions, we know that $\mu_{n,m}(k,h)=\mu_{n,m}(n-k,m-h)$. Hence Proposition \ref{prop:gamma_m} implies that 
$$ \lim_{n\to\infty} \mu_{n,m} (n-k,m)=\lim_{n\to\infty}\mu_{n,m}(k,0) = \frac 12 f_{\gamma_m}(k),~k\in\N.$$
Since we know that $f_{\gamma_m}$ is a probability mass function, it follows from the Portmanteau theorem that the QSDs for the voter model on $K_{n,m}$ converge weakly as $n\to\infty$ to the probability distribution where $S$ attains a consensus of all ``0'' or all ``1'' each with probability $1/2$, and conditioned on the opinion of $S$, the number of vertices in $L$ which are of a different opinion has probability mass function $f_{\gamma_m}$. 
\end{proof}

\section*{Acknowledgments}
The authors would like to thank an anonymous referee for their useful comments.

\bibliographystyle{alea3_mod}
\bibliography{QSD_bib}

\end{document}